\documentclass[a4paper,11pt]{article}
\usepackage{enumerate}
\usepackage{amsmath,amsthm,amssymb}
\usepackage{tikz}
\usepackage[mathlines]{lineno}

\parskip \bigskipamount

\title{Cycles in Oriented 3-graphs}
\author{Imre Leader\thanks{Department of Pure Mathematics and Mathematical Statistics, Centre for Mathematical Sciences, University of Cambridge, Wilberforce Road, Cambridge CB3 0WB, United Kingdom. Email: I.Leader@dpmms.cam.ac.uk.} \and Ta Sheng Tan\thanks{Institute of Mathematical Sciences, University of Malaya, 50603 Kuala Lumpur, Malaysia. Email: tstan@um.edu.my. This author acknowledges support received from the University Malaya Research Fund Assistance (BKP) via grant BK021-­2013.}}

\newtheorem{thm}{Theorem}[section]

\newtheorem{lemma}[thm]{Lemma}

\newtheorem{corollary}[thm]{Corollary}
\newtheorem{conjecture}[thm]{Conjecture}

\theoremstyle{remark}

\begin{document}

\maketitle
\begin{abstract}
An oriented 3-graph consists of a family of triples (3-sets), each of which is given one of its two possible cyclic orientations. A cycle in an oriented 3-graph is a positive sum of some of the triples that gives weight zero to each 2-set. 

Our aim in this paper is to consider the following question: how large can the girth of an oriented 3-graph (on $n$ vertices) be? We show that there exist oriented 3-graphs whose shortest cycle has length $\frac{n^2}{2}(1+o(1))$: this is asymptotically best possible. We also show that there exist 3-tournaments whose shortest cycle has length $\frac{n^2}{3}(1+o(1))$, in complete contrast to the case of 2-tournaments. 
\end{abstract}

\section{Introduction}

An \emph{oriented $3$-graph} on $n$ vertices consists of a family of triples (3-sets), each of which is given one of its two possible cyclic orientations. A \emph{$3$-tournament} is a complete oriented 3-graph (all triples from the ground set are oriented).

Linial and Morgenstern~\cite{linial} introduced a notion of `cycle' in an oriented 3-graph. Roughly speaking (we will give a precise definition at the start of Section 2), a cycle in an oriented 3-graph is a positive sum of some of the triples that gives weight zero to each 2-set. Linial and Morgenstern were interested in acyclic 3-tournaments (i.e. 3-tournaments not containing a cycle). They also considered cycles in higher order tournaments. (See also \cite{leader} for other results on 3-tournaments and higher order tournaments.)


Our aim in this paper is to consider the following natural question: if an oriented 3-graph (or 3-tournament) on $n$ vertices contains a cycle, how short a cycle must it contain? 
The analogous question in oriented graphs on $n$ vertices is trivial: the shortest cycle can have length $n$.
In the case of a tournament (2-tournament), it is straightforward to see that if a tournament has a (directed) cycle, it must contain a directed triangle. See Moon~\cite{moon} for background and many results on tournaments.

To be little more precise, an oriented 3-graph can be denoted by $G=(V,\mathcal{E})$, where $V$ is the vertex set (unless otherwise stated, $V=[n]=\{1,2,\ldots, n\}$ with the natural ordering) and $\mathcal{E}$ is the set of oriented triples.
Given a triple $F = \{a,b,c\}$, it can be oriented (in an oriented 3-graph) either as 

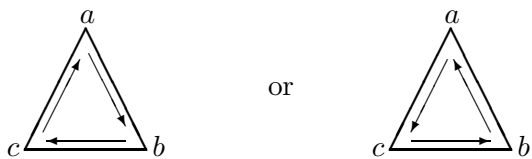
\begin{center}
\setlength{\unitlength}{0.8cm}
\begin{picture}(10,3)
  \thicklines
  \put(1,0.5){\line(1,0){2}}
  \put(2,2.5){\line(-1,-2){1}}
  \put(2,2.5){\line(1,-2){1}}
  \put(1.9,2.6){$a$}
  \put(0.7,0.4){$c$}
  \put(3.1,0.4){$b$}
  \put(5,1.4){\text{or}}
  \put(7,0.5){\line(1,0){2}}
  \put(8,2.5){\line(-1,-2){1}}
  \put(8,2.5){\line(1,-2){1}}
  \put(7.9,2.6){$a$}
  \put(6.7,0.4){$c$}
  \put(9.1,0.4){$b$}
  \thinlines
  \put(2.65,0.65){\vector(-1,0){1.3}}
  \put(2.05,2.1){\vector(1,-2){0.6}}
  \put(1.3,0.8){\vector(1,2){0.6}}

  \put(7.35,0.65){\vector(1,0){1.3}}
  \put(8.65,0.8){\vector(-1,2){0.6}}
  \put(7.95,2){\vector(-1,-2){0.6}}
\end{picture}
\end{center}

We write $\overrightarrow{abc}$ $\left(= \overrightarrow{bca} \text{ or } \overrightarrow{cab}\right )$ for the former and $\overrightarrow{acb}$ $\left(= \overrightarrow{bac} \text{ or } \overrightarrow{cba}\right )$ for the latter.
Each oriented triple induces an orientation on each of its 2-sets. Namely, if $F$ has the former orientation, we have the induced edges $\overrightarrow{ab},\overrightarrow{bc},\overrightarrow{ca}$; and if $F$ has the latter orientation, we have the induced edges $\overrightarrow{ac},\overrightarrow{cb},\overrightarrow{ba}$. A \emph{cycle} is a weighted sum (with positive weights) of triples that gives each directed edge a total weight of zero. 

For an oriented $3$-graph $G$ that contains a cycle, we are interested in the \emph{shortest cycle} in $G$, one with the smallest length. In particular, we want to know how large the girth (the length of the shortest cycle) of $G$ can be. 

The plan of the paper is as follows. We start by considering cycles in 3-tournaments. It is easy to see that if a 3-tournament contains a cycle, then its shortest cycle has length at most $\binom{n-1}{2}+1$. We do not know if the upper bound is best possible (or even asymptotically best possible), but we present a construction giving a lower bound of about $\frac{2}{3}\binom{n}{2}$. This is in complete contrast to the case of $2$-tournaments, where of course if there is a cycle, there is a cycle of length 3. Our construction is based on some embeddings of complete graphs into surfaces of high genus. This is the content of Section 2.

In Section 3 we turn our attention to general oriented 3-graphs. Here the same upper bound of $\binom{n-1}{2}+1$ applies. We show that there exists an oriented 3-graph on $n$ vertices whose shortest cycle has length $\binom{n}{2}\left(1+o(1)\right)$, which is asymptotically best possible.

For an oriented 3-graph $G$, we will usually write $V(G)$ for its vertex set and $\mathcal{E}(G)$ for its oriented triples. A triple $\{a,b,c\}$ in an oriented 3-graph is always oriented, and when its orientation is not important to us we will sometimes refer to it as $abc$.

Finally, we remark that one could view this notion of cycle as a `homological' version. Of course, as we have 3-sets but no 4-sets there is no notion of `boundary', so that there is no notion of equivalence of cycles. Our paper does not use any homological notions (but for an introduction to homology, see e.g. Armstrong~\cite{armstrong}).

\section{Cycles in $3$-tournaments}

In this section, we consider $f(n)$, the length of the shortest cycle in a 3-tournament, maximised over all 3-tournaments on $n$ vertices that contain a cycle.

It is often helpful to view cycles in matrix terms. Following \cite{linial}, the \emph{incidence matrix} of an oriented $3$-graph $G$ is an $\binom{n}{2} \times |\mathcal{E}|$ matrix $A$ whose rows and columns correspond to all 2-sets $[n]^{(2)} = \left\{i<j:i,j\in [n]\right\}$ and all oriented triples of $\mathcal{E}$ respectively. For $E=\{i<j\}$, the $(E,F)$ entry of $A$ is 
\begin{equation*}
(A)_{E,F} =
\begin{cases}
 1	&\text{ if $E\subset F$ and $F$ induces $\overrightarrow{ij}$ on $E$,}\\
 -1	&\text{ if $E\subset F$ and $F$ induces $\overrightarrow{ji}$ on $E$,}\\
 0	&\text{ otherwise.}
\end{cases}
\end{equation*}
So each column of $A$ has exactly three non-zero entries.

Given an oriented 3-graph $G=(V,\mathcal{E})$ with its incidence matrix $A$, a non-empty subset $\mathcal{C}$ of $\mathcal{E}$ is called a \emph{cycle} if there exists positive real number $\alpha_F$ for every $F\in\mathcal{C}$ such that \[\sum_{F\in\mathcal{C}} \alpha_F \mathbf{x}_F=\mathbf{0},\]
where $\mathbf{x}_F$ is the column vector of $A$ that corresponds to the oriented triple $F$. The length of the cycle $\mathcal{C}$ is the number of elements in $\mathcal{C}$. For example, the 3-tournament $\left([4], \left\{\overrightarrow{123},\overrightarrow{142}, \overrightarrow{134},\overrightarrow{243}\right\}\right)$ is itself a cycle of length four. (Note that this is called a \emph{directed $4$-set} in \cite{leader}.)

We first present an easy upper bound on $f(n)$ using standard results from linear algebra. Recall that Carath\'{e}odory's theorem says that if a point $\mathbf{x}\in \mathbb{R}^d$ lies in the convex hull of a set of points $P$, there is a subset $P'$ of $P$ consisting of at most $d+1$ points such that $\mathbf{x}$ lies in the convex hull of $P'$. 

For a 3-tournament on $n$ vertices $T$, the column vectors of its incidence matrix span a subspace of $\mathbb{R}^{\binom{n}{2}}$, and we denote this subspace by $S_T$.

\begin{lemma}\label{lemma_caratheodory}
 Let $T$ be a $3$-tournament that contains a cycle. Suppose that $S_T$ has rank $d$. Then the shortest cycle in $T$ has length at most $d+1$. 
\end{lemma}

\begin{proof}
 $S_T$ is isomorphic to $\mathbb{R}^d$. A cycle in $T$ corresponds to a set of points (column vectors) $P$,  such that its convex hull contains the origin. So by Carath\'{e}odory's theorem, there is a subset $P'$ of $P$ consisting of at most  $d+1$ points such that the origin lies in the convex hull of $P'$, which in turn corresponds to a cycle in $T$ whose length is $|P'|\le d+1$. 
\end{proof}

Together with the fact that $S_T$ has rank at most $\binom{n}{2}$ for $T$ a 3-tournament on $n$ vertices, we can deduce that $f(n)\le \binom{n}{2} + 1$ from the above lemma. The following easy result gives a better bound for the rank of $S_T$ and hence a better upper bound of $f(n)$. 

\begin{lemma}\label{lemma_rank}
 Let $T$ be a $3$-tournament on $n$ vertices. Then $S_T$ has rank at most $\binom{n}{2}-n+1$.
\end{lemma}

\begin{proof}
Let $A$ be the incidence matrix of $T$. We show that there are $n-1$ linearly independent vectors of $\mathbb{R}^{\binom{n}{2}}$ such that each one of them is orthogonal to every colomn vector of $A$. And the conclusion of the lemma follows easily from the rank-nullity theorem.

For $1\le i\le n-1$, let $\mathbf{x}_i$ be vectors of length $\binom{n}{2}$ indexed by $[n]^{(2)}$ with the following $jk$-th entries ($j<k$).
\[(\mathbf{x}_i)_{jk}\begin{cases}
 1 &\text{ if }i=j<k,\\
 -1 &\text{ if }j<k=i,\\
 0 &\text{ otherwise.}
\end{cases}\]
It is easy to see that $\mathbf{x}_1, \mathbf{x}_2, \ldots, \mathbf{x}_{n-1}$ are linearly independent as $\mathbf{x}_i$ is the only vector with non-zero $in$-th entry among them.

Given a column vector $\mathbf{v}$ of $A$, it corresponds to the orientation of a 3-set in $T$, say the 3-set $\{a<b<c\}$. The only nonzero entries of $\mathbf{v}$ are $\mathbf{v}_{ab}, \mathbf{v}_{bc}$, and $\mathbf{v}_{ac}$. Either $\mathbf{v}_{ab} = 1, \mathbf{v}_{bc} = 1, \mathbf{v}_{ac} = -1$ or $\mathbf{v}_{ab}=-1, \mathbf{v}_{bc} = -1, \mathbf{v}_{ac}=1$. In both cases, it is straightforward to check that $\mathbf{v}$ is orthogonal to $\mathbf{x}_i$ for every $i$. This completes the proof of the lemma.
\end{proof}

Combining Lemma~\ref{lemma_caratheodory} and Lemma~\ref{lemma_rank}, we have the following upper bound of $f(n)$.

\begin{corollary}\label{cycle_tournament_upper}
 The shortest cycle in a $3$-tournament on $n$ vertices that contains a cycle has length at most $\binom{n}{2}-n+2$. That is, $f(n)\le \binom{n}{2}-n+2=\binom{n-1}{2}+1$.\qed
\end{corollary}

We remark that the bound in Lemma~\ref{lemma_rank} is asymptotically best possible. (See the remark at the end of Section 3.)

We now turn our attention to the lower bound of $f(n)$. We will give a construction of a 3-tournament on $n$ vertices whose shortest cycle has length exactly $\frac{1}{3}n(n-1)$, for infinitely many $n$.
Our proof is based on some embeddings of the complete graphs in high genus surfaces. 

We will also make use of the following lemma by Linial and Morgenstern~\cite{linial}, which is particularly helpful in our construction. For the sake of completeness, we will include the proof here. 

\begin{lemma}[\cite{linial}] \label{liniallemma}
 Let $C$ be an oriented $3$-graph with the following properties.
 \begin{enumerate}[(i)]
  \item The only cycle in $C$ consists of all of its triples.
  \item No additional cycle can be created by addition of any single oriented $3$-set.
 \end{enumerate}
 Then we can orient the remaining $3$-sets (namely, the $3$-sets from $V(C)^{(3)}\setminus \mathcal{E}(C)$) to obtain a tournament $T$ such that $C$ remains as the only cycle in $T$.
\end{lemma}

\begin{proof}
 We will show that such $T$ can be constructed by orienting the remaining $3$-sets one by one.
 Let $F'$ be a 3-set which was not oriented yet. Suppose that both orientations of $F'$ give rise to new cycles. That is, $\sum \alpha_F \mathbf{x}_F  + \mathbf{x}_{F'}=\mathbf{0}$ and $\sum \alpha_F' \mathbf{x}_F  - \mathbf{x}_{F'}=\mathbf{0}$, where $\mathbf{x}_{F'}$ corresponds to $F'$ oriented one of the two ways. Then $\sum (\alpha_F +\alpha_F')\mathbf{x}_F=\mathbf{0}$ is another cycle, which does not involve $F'$. Hence this must be the only cycle $C$, implying the new cycles created use only the 3-sets from $C$ and $F'$, contradicting the properties of $C$ in the lemma. 
\end{proof}

We are now ready to construct 3-tournaments whose shortest cycle has length $\frac{2}{3}\binom{n}{2}$.

\begin{thm}\label{cycle_tournament_lower}
 Let $n\ge 4$ and $n \equiv 0,3,4 \text{or }7\pmod{12}$. Then there is a $3$-tournament on $n$ vertices whose shortest cycle has length $\frac{1}{3}n(n-1)$.
\end{thm}

\begin{proof}
It is well known that a complete graph $K_n$ can be embedded in an orientable surface of sufficiently large genus (see, for example, \cite{ringel}). In the cases where $n \equiv 0,3,4 \text{or }7 \pmod{12}$, the genus may be chosen such that the embeddings are triangulations. Given any such triangulation, we can induce an oriented 3-graph $C$, which is a cycle of length $\frac{1}{3}n(n-1)$, by orienting every face (a 3-set) on the surface in the same orientation: all oriented clockwise or all oriented anticlockwise, viewing from outside the surface. 

 
 We first claim that $C$ does not contain a cycle of a shorter length. Suppose $C'\subset C$ is a cycle. Pick any vertex $v$ in $C'$ and name the remaining vertices $v_1,v_2,\ldots,v_{n-1}$ such that the oriented 3-sets containing $v$ in $C$ are $vv_iv_{i+1}, i\in [n-1]$. (The subscripts are taken mod $n-1$.) By the definition of $C$, it is not too hard to see that if any of these 3-sets is in $C'$, all of them must be in $C'$. Indeed, $vv_{i-1}v_{i}$ and $vv_iv_{i+1}$ are the only two oriented 3-sets in $C$ containing the the 2-set $vv_i$.
 This implies that $C'$ contains all of the $n$ vertices. Repeating the above arguments with $v$ replaced by each of $v_i$ shows that all 3-sets of $C$ are in $C'$, proving the claim. 
 
 Next, we claim that $C$ has the property that no additional cycle can be created by the addition of any single oriented 3-set. 
 Suppose $C' \cup \rho$ is a cycle, where $C'\subset C$ and $\rho$ is an oriented 3-set not in $C$. That is, there exist positive coefficients $\alpha_F$ such that $\left(\sum_{F\in C'} \alpha_F \mathbf{x}_F\right)  + \mathbf{x}_\rho=\mathbf{0}$. As $n\ge 4$, there is a vertex in $C'$ but not in $\rho$. Pick any such vertex $v$, every 3-set containing $v$ must also be in the new cycle (with the same coefficient). Since $\rho \notin C$, all 3-sets of $C$ are in $C'$ and $\alpha_F$ is constant. This immediately implies that $C' \cup \rho = C\cup \rho$ is not a cycle.
 
 Now $C$ satisfies the properties in Lemma~\ref{liniallemma}, and so there is a tournament such that $C$ remains as the only cycle (hence the shortest cycle). This completes the proof of the theorem.
\end{proof}

Combining Corollary~\ref{cycle_tournament_upper} and Theorem~\ref{cycle_tournament_lower} we have $\frac{1}{3}n(n-1)\le f(n)\le \binom{n-1}{2}+1$ for $n \equiv 0,3,4 \text{or }7\pmod{12}$. Observe that each 2-set is contained in exactly two 3-sets (each giving a different orientation to the 2-set) in the shortest cycle of the tournament in Theorem~\ref{cycle_tournament_lower}, and for the exact value $f(n)$ to be closer to the upper bound, most 2-sets would have to be in three 3-sets of a shortest cycle, which we believe is unlikely. In fact, we believe that our construction is best possible.  

\begin{conjecture}
 For $n\ge 4$ and $n \equiv 0,3,4 \text{or }7\pmod{12}$, we have $f(n)=\frac{1}{3}n(n-1)$.
\end{conjecture}

For the case when $n \not\equiv 0,3,4 \text{or }7\pmod{12}$, consider the cycle $C$ (an oriented 3-graph spanning $m$ vertices) induced by the triangulation of $K_m$ as before, where $m$ is the largest integer smaller than $n$ such that $m \equiv 0,3,4 \text{or }7\pmod{12}$. Then by almost identical arguments in the proof of Theorem~\ref{cycle_tournament_lower}, there is a tournament on $n$ vertices such that $C$ is the shortest cycle. 
This, together with Corollary~\ref{cycle_tournament_upper} and Theorem~\ref{cycle_tournament_lower}, we can bound $f(n)$ for all $n\ge 4$.

\begin{corollary}
 For $n\ge 4$, we have $\frac{1}{3}m(m-1)\le f(n)\le \binom{n-1}{2}+1$, where $m$ is the largest integer smaller or equal to $n$ such that $m \equiv 0,3,4 \text{or }7\pmod{12}$.
\end{corollary}

Again, we believe that a construction attaining the upper bound is unlikely and that our construction is asymptotically best possible. 

\begin{conjecture}
$f(n)=\left(\frac{1}{3}+o(1)\right)n^2$ for all $n\ge 4$.
\end{conjecture}

\section{Cycles in oriented 3-graphs}

Suppose $G$ is an oriented $3$-graph on $n$ vertices and $G$ contains a cycle. Using the exact same arguments as those used to derive Corollary~\ref{cycle_tournament_upper}, we know that the shortest cycle in $G$ has length at most $\binom{n-1}{2}+1$.

Our main aim in this section is to show that there exists an oriented 3-graph $G$ on $n$ vertices such that the only cycle in $G$ consists of all the triples of $G$, and has length $\binom{n}{2}(1-o(1))$, attaining the upper bound asymptotically. 
Our construction could be viewed as an attempt to ``add as many projective planes as possible to a small starting configuration''.

By a \emph{single cycle}, we shall mean an oriented 3-graph $G$ where the only cycle in $G$ consists of all its triples. 
(Note that it follows from this that all cycles in $G$ are the same, up to multiplication by positive reals.)
The basic idea is to start from a base single cycle $P$, then delete a triple from it, and attach $P$ to another single cycle in which a triple is removed. This will result another single cycle, which we can again attach $P$ onto.  

To give a better insight to how our final construction is obtained, we will first start with a simple and symmetric base single cycle, the projective plane. (We assume no knowledge of the projective plane; however, for background on surfaces and the projective plane, see e.g. Armstrong~\cite{armstrong} or Hatcher~\cite{hatcher}.) This will end up giving a single cycle $G$ of length $(\frac{2}{3}-o(1))\binom{n}{2}$.
Then we will see how to modify this construction to give girth of $(1-o(1))\binom{n}{2}$.

For convenience, we will refer a set of triples $\mathcal{F}$ as a \emph{star system} if there is a vertex $a$ such that $F\cap F' = \{a\}$ for every pair of distinct $F,F'$ in $\mathcal{F}$. 

\subsection{Attaching the projective plane}

Consider the standard 6-point triangulation of the projective plane in Fig. 1, where each of the 10 faces (triples) is oriented clockwise. 
Adding the triple $\{x,y,z\}$ with anticlockwise orientation $\overrightarrow{xzy}$ results in a single cycle (the triple $xyz$ has coefficient 2, while the other triples each has coefficient 1). Now delete the triple $\overrightarrow{acb}$ and we will denote this oriented 3-graph on six vertices by $P$ for the rest of the paper.
That is, $V(P)=\{x,y,z,a,b,c\}$ and $\mathcal{E}(P)=\{\overrightarrow{xya},\overrightarrow{ayz},\overrightarrow{azc},\overrightarrow{czx},\overrightarrow{cxy},\overrightarrow{cyb},\overrightarrow{byz},\overrightarrow{bzx},\overrightarrow{bxa},\overrightarrow{xzy}\}$.

\begin{center}
\begin{tikzpicture}
 \draw[thick] (0,1.1) node[anchor = north]{a} -- (4/3,-0.9) -- (-4/3,-0.9) -- cycle; 
 \node at (-1,-0.67) {$b$};
 \node at (1,-0.67) {$c$};
 \draw[thick] (0,0) circle (3cm);
 
 \draw[thick] (0,3) node[anchor = south]{$y$} -- (0,1.1);
 \draw[thick] (3,0) arc (0:150:3cm) node[anchor = south east]{$x$} -- (0,1.1);
 \draw[thick] (3,0) arc (0:150:3cm) -- (-4/3,-0.9);
 \draw[thick] (3,0) arc (0:210:3cm) node[anchor = north east]{$z$} -- (-4/3,-0.9);
 \draw[thick] (4/3,-0.9) -- (0,-3) node[anchor = north]{$y$} -- (-4/3,-0.9); 
 \draw[thick] (3,0) arc (0:210:3cm) -- (-4/3,-0.9);
 \draw[thick] (3,0) arc (0:330:3cm) node[anchor = north west]{$x$}-- (4/3,-0.9);
 \draw[thick] (3,0) arc (0:30:3cm) node[anchor = south west]{$z$} -- (4/3,-0.9);
 \draw[thick] (3,0) arc (0:30:3cm) -- (0,1.1);
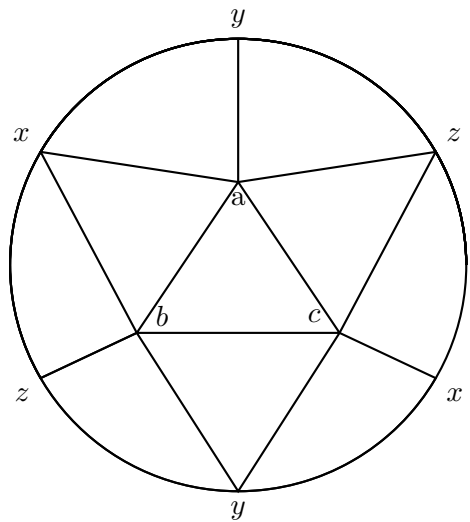
 \node at (0,-4) {\textbf{Fig.1 }  The standard 6-point triangulation of the real projective plane};
\end{tikzpicture}
\end{center}

Given an oriented 3-graph $G$ and an oriented triple $F=\overrightarrow{ijk}$ in $G$, we can form a new oriented 3-graph by attaching $P$ on $G$ via $F$. More precisely, we say $G'$ is a \emph{$(P,F)$-attachment on $G$} (or simply \emph{$P$-attachment on $G$} when $F$ is understood) where $G'$ has vertex set $V(G')=V(G)\cup V(P)$ with $i,j,k$ identified with $a,b,c$ respectively, and set of triples $\mathcal{E}(G') =\mathcal{E}(G)\cup \mathcal{E}(P) \setminus \left\{\overrightarrow{ijk}\right\}$. 
So if $G$ is an oriented 3-graph on $k$ vertices and has $l$ triples, then the $P$-attachment on $G$ is an oriented 3-graph on $k+3$ vertices and has $l+9$ triples.
It is straightforward to check that if $G$ is a single cycle, then $G'$ is also a single cycle. (Or see the proof of Lemma~\ref{p-attachment}.)

We can also attach $P$ on an oriented 3-graph via a set of oriented triples, namely, a star system, as follows.
Let $G$ be an oriented 3-graph on $k$ vertices and $\mathcal{F}=\left\{F_i=\overrightarrow{ab_ic_i}:i=1,2,\ldots,d\right\}$ be a star system in $G$. We will attach $P$ on $G$ via $F_i$ one by one. Set $G_1$ to be the $(P,F_1)$-attachment on $G$, where $G_1$ has vertex set $V(G)\cup \{x_1=x,y_1=y,z_1=z\}$. 
Now, suppose $G_{i-1}$ is constructed, let $G_i$ be the \emph{modified $(P,F_i)$-attachment of $G_{i-1}$}: identify the three new vertices of $(P,F_i)$-attachment of $G_{i-1}$ with $\{x,y,z\}$ and delete any repeated triples. That is, if $x_i,y_i,z_i$ are the three new vertices of $(P,F_i)$-attachment of $G_{i-1}$, we identify $x_i$ with $x$, $y_i$ with $y$ and $z_i$ with $z$. 
(Note that at each stage, we always attach $P$ with the preserved orientations, for example, the triples $\{x_i,y_i,z_i\}$ has orientation $\overrightarrow{x_iz_iy_i}$. So $G_i$ is well defined for all $i$.) And finally, we say $G'=G_d$ is the \emph{$(P,\mathcal{F})$-attachment on $G$} 
(or simply \emph{$P$-attachment on $G$} when $\mathcal{F}$ is understood) on $k+3$ vertices. 

In other words, for $i\ge 2$, $G_i$ is obtained from $G_{i-1}$ by deleting the triple $F_i=\overrightarrow{ab_ic_i}$ and adding the set of triples $\left\{\overrightarrow{azc_i},\overrightarrow{c_izx},\overrightarrow{c_ixy},\overrightarrow{b_ic_iy},\overrightarrow{b_iyz},\overrightarrow{b_izx},\overrightarrow{ab_ix}\right\}$.

It is not too hard to see that if $G$ is a cycle, then the $P$-attachment on $G$ is also a cycle. In fact, if $G$ is a single cycle with a star system $\mathcal{F}$ of $l$ triples, the $(P,\mathcal{F})$-attachment on $G$ is also a single cycle.

\begin{lemma}\label{p-attachment}
 Let $G$ be a single cycle (an oriented $3$-graph) on $k$ vertices of length $l$. Suppose $\mathcal{F}=\left\{F_i = \overrightarrow{ab_ic_i}: i=1,2,\ldots, d, b_i\ne c_j \text{ for all }i,j\right\}$ is a star system in $G$. Then the $(P,\mathcal{F}$)-attachment on $G$, $G'$, is a single cycle on $k+3$ vertices of length $l+6d+3$. Furthermore, $G'$ contains a star system of size $d+1$.
 \end{lemma}

 \begin{proof}
  Let $G$, $\mathcal{F}$ and $G'$ be as in the lemma. We first show that $\mathcal{E}(G')$ is a cycle. Let $A$ be the incidence matrix of $G$. Since $G$ is a single cycle, there exists positive $\alpha_F$ for every triple $F$ in $G$ such that 
  \[\sum_{F\in\mathcal{E}(G)} \alpha_F \mathbf{x}_F=\mathbf{0},\]
  where $\mathbf{x}_F$ is the column vector of $A$ that corresponds to the oriented triple $F$.
  
  Now let $\mathcal{P}_i=\left\{\overrightarrow{azc_i},\overrightarrow{c_izx},\overrightarrow{c_ixy},\overrightarrow{b_ic_iy},\overrightarrow{b_iyz},\overrightarrow{b_izx},\overrightarrow{ab_ix}\right\}$, $\mathcal{X} = \left\{\overrightarrow{xya},\overrightarrow{ayz}\right\}$ and $R=\overrightarrow{xzy}$. 
  With a slight abuse of notation, we now refer $\mathbf{x}_F$ to be the column vector of the incidence matrix of $G'$ that corresponds to the triple $F$ in $G'$.
  It is then straightforward to see that 
  \[\left(\sum_{F\in\mathcal{E}(G)\setminus \mathcal{F}} \alpha_F \mathbf{x}_F \right)+ \sum_{i=1}^{d} \left(\alpha_{F_i}\sum_{F\in\mathcal{P}_i} \mathbf{x}_F\right)+\left(\sum_{i=1}^{d} \alpha_{F_i}\right)\left(2\mathbf{x}_R+\sum_{F\in\mathcal{X}} \mathbf{x}_F\right)=\mathbf{0}.\]

  By construction, the $P$-attachment $G'$ is a cycle on $k+3$ vertices and has $l+6d+3$ triples. Indeed, $G_1$, the $(P,T_1)$-attachment on $G$, has $l+9$ triples, and $G_i$, the modified $(P,F_i)$-attachment on $G_{i-1}$, has six additional triples.
  Also, note that the set of triples $\mathcal{F}'=\left\{\overrightarrow{yb_ic_i}:i=1,2,\ldots,d\right\}\cup \left\{\overrightarrow{yxz}\right\}$ in $G'$ is a star system of size $d+1$.
    
  So we only need to show that $G'$ is a single cycle, that is, the only cycle in $G'$ consists of all the triples in $G'$. We will do this by showing that $G_i$ is a single cycle for each $i\in \{1,2,\ldots, d\}$. 
  Let $G_0 = G$ and we note that for $i\ge 1$, $G_i$ is obtained from $G_{i-1}$ by deleting the triple $F_i=\overrightarrow{ab_ic_i}$ and adding the set of triples $\mathcal{S}$. Here,
  $\mathcal{S}=\left\{\overrightarrow{azc_1},\overrightarrow{c_1zx},\overrightarrow{c_1xy},\overrightarrow{b_1c_1y},\overrightarrow{b_1yz},\overrightarrow{b_1zx},\overrightarrow{ab_1x},\overrightarrow{xya},\overrightarrow{ayz},\overrightarrow{xzy}\right\}$ for $i=1$ and 
  $\mathcal{S}=\left\{\overrightarrow{azc_i},\overrightarrow{c_izx},\overrightarrow{c_ixy},\overrightarrow{b_ic_iy},\overrightarrow{b_iyz},\overrightarrow{b_izx},\overrightarrow{ab_ix}\right\}$ for $i\ge 2$.
  
  Now, for $i\ge 1$, suppose that $G_{i-1}$ is a single cycle and $\mathcal{C}$ is a cycle in $G_i$. If $\mathcal{C}\cap \mathcal{S}=\emptyset$, then $\mathcal{C}\subset \mathcal{E}(G_{i-1})$ is a cycle of length strictly shorter than $|\mathcal{E}(G_{i-1})|$ as $F_i\notin \mathcal{C}$, contradicting $G_{i-1}$ is a single cycle. So we may assume $\mathcal{C}$ contains at least one triple in $\mathcal{S}$, and this will imply that $\mathcal{C} \supset \mathcal{S}$.
  This is because the only two triples that contain the 2-set $b_ix$ (also the two sets $b_iz,b_iy,c_iy,c_ix,c_iz$) are both in $\mathcal{S}$, inducing opposite directions of $b_ix$. And if $b_ix$ is contained in a triple in $\mathcal{C}$, both these triples must be in $\mathcal{C}$. Using similar arguments (by considering the 2-sets $ax,ay$ and $az$), we can further claim that the triples $\overrightarrow{xya},\overrightarrow{ayz}$ and $\overrightarrow{xzy}$ are also in $\mathcal{C}$.
  
  We can then write $\mathcal{C} = \mathcal{C}'\cup \mathcal{S}$, where $\mathcal{C}'\subset \mathcal{E}(G_{i-1})$ and $\mathcal{C}'\supset \left\{\overrightarrow{xya},\overrightarrow{ayz},\overrightarrow{xzy}\right\}$. It is then straightforward to check that $\mathcal{C}'\cup \{F_i\}$ is a cycle in $G_i$, and hence a cycle in $G_{i-1}$. Since $G_{i-1}$ is a single cycle, necessarily $\mathcal{C}'\cup \{F_i\} = \mathcal{E}(G_{i-1})$, implying $\mathcal{C} = \mathcal{E}(G_i)$, completing the proof of the lemma.
 \end{proof}

By repeatedly applying Lemma~\ref{p-attachment} to a single cycle, we can construct a single cycle with increasing length.

\begin{corollary}\label{p-bound}
 There exists an oriented $3$-graph on $n$ vertices whose shortest cycle has length $\frac{2}{3}\binom{n}{2}(1+o(1))$.
\end{corollary}

\begin{proof}
 Let $G_0$ be a single cycle on $4$ vertices with $\mathcal{E}(G_0)=\left\{\overrightarrow{123},\overrightarrow{142},\overrightarrow{134},\overrightarrow{243}\right\}$ and let $\mathcal{F}_0=\left\{\overrightarrow{123}\right\}$. For $i\ge 1$, suppose $G_{i-1}$ is a single cycle containing a star system $\mathcal{F}_{i-1}$. Then by Lemma~\ref{p-attachment}, there exists a single cycle $G_i$ containing a star system $\mathcal{F}_i$ of size $|\mathcal{F}_{i-1}|+1$.
 
 By construction, $G_i$ has $4+3i$ vertices, $|\mathcal{F}_i|=i+1$, and $G_i$ is a single cycle of length
 \begin{align*}
  |\mathcal{E}(G_{i})|=&|\mathcal{E}(G_{i-1})|+6|\mathcal{F}_{i-1}|+3\\
		      =&\left(|\mathcal{E}(G_{i-2})|+6|\mathcal{F}_{i-2}|+3\right)+6|\mathcal{F}_{i-1}|+3\\
		      & \vdots\\
		      =&|\mathcal{E}(G_{0})|+6(|\mathcal{F}_{0}|+|\mathcal{F}_{1}|+\ldots+|\mathcal{F}_{i-1}|)+3i\\
		      =&4+6(1+2+\ldots+i)+3i\\
		      =&3i^2+6i+4.
 \end{align*}

 Now, letting $G=G_k$, we see that $G$ is a single cycle on $n=4+3k$ vertices of length \[\frac{n^2-2n+4}{3}=\frac{2}{3}\binom{n}{2}(1-o(1)).\] 
\end{proof}

\subsection{Attaching a modified projective plane} 
 
In order to improve the lower bound in Corollary~\ref{p-bound}, we can try to use a better base single cycle in the inductive construction. 
Very strangely, it turns out that this will lead to a much improved construction.

By an increment of 3 vertices, Lemma~\ref{p-attachment} produces a larger single cycle, as well as a star system with one extra triple. It would be better if we could have a base single cycle where the new single cycle produced has less than 6 extra vertices and the star system is enlarged by two extra triples. 
 
\begin{center}
\begin{tikzpicture}
 \draw[thick] (0,1.1) node[anchor = north]{$a$} -- (4/3,-0.9) -- (-4/3,-0.9) -- cycle; 
 \node[thick] at (-1,-0.67) {$b$};
 \node[thick] at (1,-0.67) {$c$};
 \draw[thick] (0,0) circle (3cm);
 
 \draw[thick] (0,3) node[anchor = south]{$y$} -- (0,1.1);
 \draw[thick] (0,3) -- (-1.2,1.1) node[anchor = north west] {$t_1$} -- (0,1.1);
 \draw[thick] (3,0) arc (0:30:3cm) -- (1.4,0.65) node[anchor = north east] {$t_2$};
 \draw[thick] (0,1.1) -- (1.4,0.65) -- (4/3,-0.9);
 \draw[thick] (3,0) arc (0:150:3cm) node[anchor = south east]{$x$} -- (-1.2,1.1) -- (-4/3,-0.9);
 \draw[thick] (3,0) arc (0:150:3cm) -- (-4/3,-0.9);
 \draw[thick] (3,0) arc (0:210:3cm) node[anchor = north east]{$z$} -- (-4/3,-0.9);
 \draw[thick] (4/3,-0.9) -- (0,-3) node[anchor = north]{$y$} -- (-4/3,-0.9); 
 \draw[thick] (3,0) arc (0:210:3cm) -- (-4/3,-0.9);
 \draw[thick] (3,0) arc (0:330:3cm) node[anchor = north west]{$x$}-- (4/3,-0.9);
 \draw[thick] (3,0) arc (0:30:3cm) node[anchor = south west]{$z$} -- (4/3,-0.9);
 \draw[thick] (3,0) arc (0:30:3cm) -- (0,1.1);
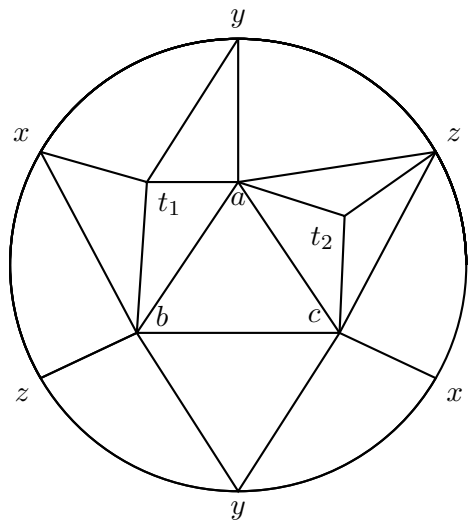
 \node at (0,-4) {\textbf{Fig.2 }  The modified triangulation of the real projective plane};
\end{tikzpicture}
\end{center}
 
Consider the modified triangulation of the projective plane in Fig. 2, where each of the 14 faces (triples) are oriented clockwise. Adding the triple $\{x,y,z\}$ with anticlockwise orientation $\overrightarrow{xzy}$ gives a cycle  (in fact a single cycle, where the triple $xyz$ has coefficient 2, while the other triples each has coefficient 1). Now delete the triple $\overrightarrow{acb}$ and identify $t_1$ and $t_2$ (denote by $t_1=t_2=t$) to obtain the oriented 3-graph $S$. It is straightforward to check that $S$ is a single cycle with a triple removed.

So $S$ is an oriented 3-graph on 7 vertices and has 14 triples, where $V(S)=\{x,y,z,t,a,b,c\}$ and $\mathcal{E}(S)=\{\overrightarrow{xyt},\overrightarrow{aty},\overrightarrow{ayz},\overrightarrow{azt},\overrightarrow{atc},\overrightarrow{ctz},\overrightarrow{czx},\overrightarrow{cxy},\overrightarrow{cyb},\overrightarrow{byz},\overrightarrow{bzx},\overrightarrow{bxt},\overrightarrow{bta},\overrightarrow{xzy}\}$.

For an oriented 3-graph $G$ that contains a triple $F=\overrightarrow{ijk}$, we can define the \emph{$(S,F)$-attachment on $G$} as in the previous subsection. That is, we say $G'$ is the \emph{$(S,F)$-attachment on $G$} (or simply \emph{$S$-attachment on $G$} when $F$ is understood) where $G'$ has vertex set $V(G')=V(G)\cup V(S)$ with $i,j,k$ identified with $a,b,c$ respectively, and set of triples $\mathcal{E}(G') =\mathcal{E}(G)\cup \mathcal{E}(S) \setminus \left\{\overrightarrow{ijk}\right\}$. 
So if $G$ is an oriented 3-graph on $k$ vertices and has $l$ triples, then the $S$-attachment on $G$ is an oriented 3-graph on $k+4$ vertices and has $l+13$ triples.
 
Similarly, we define the \emph{$(S,\mathcal{F})$-attachment on $G$} for an oriented 3-graph $G$ on $k$ vertices, where $\mathcal{F}=\left\{F_i=\overrightarrow{ab_ic_i}:i=1,2,\ldots,d\right\}$ is a star system in $G$, as follows. Let $G_1$ be the $(S,F_1)$-attachment on $G$. 
And for $i\ge 2$, $G_i$ is obtained from $G_{i-1}$ by deleting the triple $F_i=\overrightarrow{ab_ic_i}$ and adding the triples $\left\{\overrightarrow{atc_i},\overrightarrow{c_itz},\overrightarrow{c_izx},\overrightarrow{c_ixy},\overrightarrow{c_iyb_i},\overrightarrow{b_iyz},\overrightarrow{b_izx},\overrightarrow{tb_ix},\overrightarrow{b_ita}\right\}$. Then $G_d$ is the $(S,\mathcal{F})$-attachment on $G$.
 
Given a single cycle with a set of triples with a certain property - a star system - we can attach $S$ inductively in such a way that each $S$-attachment produces a larger single cycle that contains a larger star system. The following lemma, which is very similar to Lemma~\ref{p-attachment}, is the key method in our inductive construction. The proof is similar to that of Lemma~\ref{p-attachment} (with extra details) and so is omitted. 
 
\begin{lemma}\label{s-attachment}
 Let $G$ be a single cycle (an oriented $3$-graph) on $k$ vertices of length $l$. Suppose $\mathcal{F}=\left\{F_i = \overrightarrow{ab_ic_i}: i=1,2,\ldots, d, b_i\ne c_j \text{ for all }i,j\right\}$ is a star system in $G$. Then the $(S,\mathcal{F}$)-attachment on $G$, $G'$, is a single cycle on $k+4$ vertices of length $l+8d+5$. Furthermore, $G'$ contains a star system of size $d+2$. \qed
 \end{lemma}
 
By repeatedly applying Lemma~\ref{s-attachment}, we now obtain a single cycle of length $\binom{n}{2}(1+o(1))$. By our earlier remarks (Corollary~\ref{cycle_tournament_upper}), this is asymptotically best possible.

\begin{corollary}
  There exists an oriented $3$-graph on $n$ vertices whose shortest cycle has length $\binom{n}{2}(1+o(1))$.
\end{corollary}

\begin{proof}
 
 Let $G_0$ be the single cycle on $7$ vertices obtained from adding the triple $\overrightarrow{acb}$ to $S$. It has 15 triples and a `good' set of triples $\mathcal{F}_0=\left\{\overrightarrow{ybc},\overrightarrow{yat},\overrightarrow{yxz}\right\}$.
 For $i\ge 1$, suppose $G_{i-1}$ is a single cycle containing a set of triples $\mathcal{F}_{i-1}$ satisfying the property in Lemma~\ref{s-attachment}. Then there exists a single cycle $G_i$ containing a set of triples $\mathcal{F}_i$, again satisfying the property in Lemma~\ref{s-attachment}.
 
 By construction, $G_i$ has $7+4i$ vertices, $|\mathcal{F}_i|=2i+3$, and $G_i$ is a single cycle of length
 \begin{align*}
  |\mathcal{E}(G_{i})|=&|\mathcal{E}(G_{i-1})|+8|\mathcal{F}_{i-1}|+5\\
		      =&\left(|\mathcal{E}(G_{i-2})|+8|\mathcal{F}_{i-2}|+5\right)+8|\mathcal{F}_{i-1}|+5\\
		      & \vdots\\
		      =&|\mathcal{E}(G_{0})|+8(|\mathcal{F}_{0}|+|\mathcal{F}_{1}|+\ldots+|\mathcal{F}_{i-1}|)+5i\\
		      =&15+8\left(3+\ldots+(2i+1)\right)+5i\\
		      =&8i^2+21i+15.
 \end{align*}

 Now, letting $G=G_k$, we see that $G$ is a single cycle on $n=7+4k$ vertices of length \[\frac{2n^2-7n+11}{4}=\binom{n}{2}(1-o(1)).\]
\end{proof}

We remark that the inductive construction above is very far from being an optimal single cycle in a 3-tournament. Indeed, for any single cycle $G$ and any triple $F$ in $G$, the $(S,F)$-attachment on $G$ has the property that any orientation of $F$ will give a shorter cycle.
Note also that the above construction also shows that the rank of the vector space spanned by its incidence matrix has rank at least $\binom{n}{2}(1+o(1))$, implying that the bound in Lemma~\ref{lemma_rank} is asymptotically best possible.

\section{Concluding remarks}

In this paper we have addressed of the girth of 3-tournaments and oriented 3-graphs. Linial and Morgenstern~\cite{linial} also considered cycles in higher order tournaments: $d$-tournaments for general $d$. (See \cite{linial} for relevant definitions.) It would be interesting to know how girth behaves there. 
There is again a linear algebra bound of $\binom{n}{d-1}+1$: how close is this to being attained?

Finally, although these questions arose naturally in the context of oriented 3-graphs and $d$-graphs, it would be interesting to know what happens in the undirected case.

\end{document}